\documentclass[12pt, regno]{amsart}

\usepackage{amssymb}
\usepackage{amscd}
\usepackage{amsfonts}
\usepackage{setspace}
\usepackage{version}
\usepackage{endnotes}
\usepackage{float}
\usepackage{color}

\newtheorem{theorem}{Theorem}[section]
\newtheorem{lemma}[theorem]{Lemma}
\newtheorem{proposition}[theorem]{Proposition}

\newtheorem{corollary}[theorem]{Corollary}

\theoremstyle{remark}
\newtheorem{remark}[theorem]{Remark}

\theoremstyle{definition}

\theoremstyle{remark}

\numberwithin{equation}{section}

\newcommand{\summ}{\mathop{\sum\sum}}
\newcommand\numberthis{\stepcounter{equation}\tag{\theequation}}

\newcommand{\R}{\mathbb{R}}

\newcommand{\ee}{\varepsilon}
\newcommand{\1}{\mathbf{1}}

\newcommand{\dd}{{\rm d}}

\newcommand{\vphi}{{\varphi}}

\newcommand{\cu}{{\, \mathfrak u}}
\newcommand{\ssum}[1]{\sum_{\substack{#1}}}
\newcommand{\cond}{{\rm cond}}
\newcommand{\prim}{\text{ primitive}}

\begin{document}

\title[Smooth-supported multiplicative functions in APs]{Smooth-supported multiplicative functions in arithmetic progressions beyond the $x^{1/2}$-barrier}

\author[S. Drappeau]{Sary Drappeau}
\address{SD: Aix Marseille Universit\'e, CNRS, Centrale Marseille \\ I2M UMR 7373 \\ 13453 Marseille \\ France}

\email{sary-aurelien.drappeau@univ-amu.fr}

\author[A. Granville]{Andrew Granville}
\address{AG: D\'epartement de math\'ematiques et de statistique\\
Universit\'e de Montr\'eal\\
CP 6128 succ. Centre-Ville\\
Montr\'eal, QC H3C 3J7\\
Canada; 
and Department of Mathematics \\
University College London \\
Gower Street \\
London WC1E 6BT \\
England.
}
\thanks{A.G.~has received funding from the
European Research Council  grant agreement n$^{\text{o}}$ 670239, and from NSERC Canada under the CRC program.}
\email{{\tt andrew@dms.umontreal.ca}}

\author[X. Shao]{Xuancheng Shao }
\address{XS: Mathematical Institute\\ Radcliffe Observatory Quarter\\ Woodstock Road\\ Oxford OX2 6GG \\ United Kingdom}
\email{Xuancheng.Shao@maths.ox.ac.uk}
\thanks{X.S.~was supported by a Glasstone Research Fellowship.}

\thanks{We would like to thank Adam Harper for a valuable discussion}


\date{\today}

\begin{abstract}  
We show that smooth-supported multiplicative functions $f$ are well-distributed in arithmetic progressions $a_1a_2^{-1} \pmod q$ on average over moduli $q\leq x^{3/5-\varepsilon}$ with $(q,a_1a_2)=1$.
\end{abstract}

\maketitle

{\centering \emph{In memory of Klaus Roth}\par}
\medskip

\section{Introduction}

In this paper we prove a Bombieri-Vinogradov type theorem for general multiplicative functions supported on smooth numbers, with a fixed member of the residue class. Given a multiplicative function $f$, we define, whenever $(a,q)=1$,
 $$
\Delta(f,x;q,a):=\sum_{\substack{n\leq x \\ n\equiv a \pmod q}} f(n)  -  \frac 1{\varphi(q)} \sum_{\substack{n\leq x \\ (n,q)=1}} f(n).
$$
We wish to prove that, for an arbitrary fixed $A>0$, 
\begin{equation}\label{eq:BVI}
\sum_{\substack{ q\sim Q \\ (a,q)=1}}  | \Delta(f,x;q,a) | \ll \frac x {(\log x)^A}
\end{equation}
where, here and henceforth, ``$q\sim Q$'' denotes the set of integers $q$ in the range $Q<q\leq 2Q$, for as large values of $Q$ as possible.  Let
\[
F(s) = \sum_{n=1}^{\infty} \frac{f(n)}{n^{s}}\ \ \text{and} \ \
- \frac{F'(s)}{F(s)} = \sum_{n=2}^{\infty}   \frac{\Lambda_f(n)}{n^{s}},  
\]
for Re$(s)>1$.  Following \cite{GHS}, we restrict  attention to the class ${\mathcal C} $ of multiplicative functions $f$ for which
\[
 |\Lambda_f(n)|\leq   \Lambda(n) \ \ \textrm{for all} \ \ n\geq 1.
 \]
This includes  most $1$-bounded multiplicative functions of interest, including all $1$-bounded completely multiplicative functions. Two key observations are that if $f\in \mathcal C$ then each $|f(n)|\leq 1$, and if $f\in\mathcal C$ and $F(s)G(s)=1$ then $g\in \mathcal C$.

In \cite{GSh} the last two authors showed that there are two different reasons that the sum in \eqref{eq:BVI}
might be $\gg x/\log x$. First $f$ might be a character of small conductor (for example  $f(n) = (n/3)$), or might ``correlate'' with such a character; secondly $f$ might have been selected so that $f(p)$ works against us for most primes $p$ in the range $x/2<p\leq x$. We handled these potential pretentious problems as follows.

To avoid issues with the values $f(p)$ at the large primes $p$ we only allow $f$ to be supported on $y$-smooth integers\footnote{That is, integers all of whose prime factors are $\leq y$.} for $y=x^\theta$, for some small $\theta>0$. 

To avoid issues with the function $f$ correlating with a given character $\chi$, note that this happens when
\[
 S_f(X,\chi):= \sum_{n\leq X} f(n) \overline{\chi}(n) 
\]
 is ``large'' (that is, $\gg X$, or $\gg X/(\log X)^A$) for some $X$ in the range $x^{1/2}<X\leq x$, in which case  \eqref{eq:BVI} might well be false. We can either assume that this is false for all $\chi$ (which is equivalent to what is known as a ``Siegel-Walfisz criterion'' in the literature), or we can take account of such $\chi$ in the ``Expected Main Term''. We will begin by doing the latter, and then deduce the former as a corollary.

 We start by stating the Siegel-Walfisz criterion:
\medskip

\noindent \emph{The  Siegel-Walfisz criterion}: For any fixed $A > 0$, we say that $f$ satisfies the $A$-Siegel-Walfisz criterion if for any $(a,q)=1$ and any $x \geq 2$ we have the bound
\[
|\Delta(f,x;q,a)| \ll_A\frac 1 {(\log x)^A} \sum_{n \leq x} |f(n)|.
\]
We say that $f$ satisfies the Siegel-Walfisz criterion if it satisfies the $A$-Siegel-Walfisz criterion for all $A > 0$.
\medskip

For a set of primitive characters $\Xi$, let $\Xi_q$ be the  set  of those characters $\pmod q$ which are induced by the characters in $\Xi$. Then denote
\[
\Delta_\Xi(f,x;q,a) :=  \sum_{\substack{n\leq x \\ n\equiv a \pmod q}} f(n) -   \frac 1{\varphi(q)} \sum_{ \substack{  \chi  \in   \Xi_q}} \chi(a) S_f(x,\chi)
\]

In \cite{GSh} we proved the following result:

\begin{theorem} \label{MathResult3+} 
 Fix $\delta, B>0$. Let $y = x^{\varepsilon}$ for some $\varepsilon > 0$ sufficiently small in terms of $\delta$. Let $f\in \mathcal C$ be a   multiplicative function which is only supported on $y$-smooth integers. Then there exists a set, $\Xi$, of primitive characters, containing $\ll   (\log x)^{6B+7+ o(1)}$ elements, such that for any $1 \leq |a| \ll Q \leq x^{\frac{20}{39}-\delta} $, we have
 \begin{equation*}
 \sum_{\substack{q \sim Q \\ (a,q) = 1}} \left|  \Delta_{\Xi}(f,x;q,a)  \right|  \ll \frac x{(\log x)^B}.
 \end{equation*}
Moreover, if $f$ satisfies the   Siegel-Walfisz criterion then   
\[ \sum_{\substack{q \sim Q \\ (a,q)=1}}   \left|  \Delta(f,x;q,a)  \right|    \ll \frac x{(\log x)^B}. \]
\end{theorem}

 In this article we develop Theorem \ref{MathResult3+} further, 
 allowing $Q$ as well as $y$ to vary over a much wider range, and obtaining upper bounds in terms of (the more appropriate) $\Psi(x,y)$, the number of $y$-smooth integers up to $x$.

\begin{theorem}   \label{MainTheorem} Fix $\varepsilon, A>0$. Suppose that $f\in \mathcal C$, and is only supported on $y$-smooth numbers, where
 \begin{equation}\label{y-range}
x^\delta> y\geq \exp\left(   \frac 52 \cdot \frac{ \sqrt{ \log x  }  \log\log x} { \sqrt{   \log\log\log x} } \right)
 \end{equation}
for some sufficiently small $\delta > 0$. Then there exists a set, $\Xi$, of primitive characters, containing $\ll (\log x)^{6A+38}$ elements, such that if  $1\leq |a_1|, |a_2| \leq x^\delta$ then
\[
   \sum_{\substack{q\leq x^{3/5-\varepsilon} \\ (q,a_1a_2)=1}}   \left|  \Delta_\Xi(f,x;q,a_1\overline{a_2})  \right|
    \ll   \frac{\Psi(x,y)}{(\log x)^A}  .
\]
 Moreover, if $f$ satisfies the   Siegel-Walfisz criterion then   
\[
   \sum_{\substack{q\leq x^{3/5-\varepsilon} \\ (q,a_1a_2)=1}}   \left|  \Delta(f,x;q,a_1\overline{a_2})  \right|
    \ll   \frac{\Psi(x,y)}{(\log x)^A}  .
\]
 \end{theorem}

It would be interesting to extend the range \eqref{y-range} in Theorem \ref{MainTheorem} down to any $y\geq (\log x)^C$ for some large constant $C$. We discuss the main issue that forces us to restrict the range in Theorem \ref{MainTheorem}  to $y>\exp ( (\log x)^{1/2+o(1)})$ in Remark~\ref{rem:y-range}. In our proofs we have used the range $y\geq (\log x)^C$  when we can, as an aid to future research on this topic, and to make clear what are the sticking points. 

Fouvry and Tenenbaum (Th\'eor\`eme 2 in \cite{FT2}) established such a result when $f$ is the characteristic function of the $y$-smooth integers (with $y < x^{\delta}$) and $a_2=1$, in the same range $q\leq x^{3/5-\ee}$, but with the bound
$\ll x/(\log x)^A$. This was improved by Drappeau \cite{Dr1} to  $\ll \Psi(x,y)/(\log x)^A$ for 
$(\log x)^C<y\leq x^\delta$. 

The  proof  of Theorem \ref{MainTheorem} combines the ideas from our earlier articles \cite{Dr1} and \cite{GSh}. Perhaps the most innovative feature of this article, given \cite{Dr1} and \cite{GSh}, comes in Theorem \ref{thm:large-sieve-smooth} in which we prove a version of the classical large sieve inequality (towards which Roth's work~\cite{Rot} played a pivotal role) for (the notably sparse) sequences supported on the $y$-smooth numbers, which may be of independent interest.

 \section{Reduction to a larger set of exceptional moduli} 
 
 We begin by modifying estimates from  \cite{Dr1} to prove Theorem~\ref{SaryResult}, which is a version of Theorem \ref{MainTheorem} with a far larger exceptional set of characters. This is key to the proof of Theorem \ref{MainTheorem} since we now only need to cope with relatively small moduli. We therefore define  
 $\mathcal A(D)$ to be the set of all primitive characters of conductor $\leq D$.

 \begin{theorem}   \label{SaryResult} For fixed $\varepsilon, A>0$,  there exist $C,\delta>0$  such that for any $y$ in the range $(\log x)^C<y  \leq x^{\delta}$,  and any $f\in \mathcal C$  which is only supported on $y$-smooth numbers,
 we have
 \[ \sum_{\substack{q\leq x^{3/5-\varepsilon} \\ (q,a_1a_2)=1}}   \left|  \Delta_{\mathcal A}(f,x;q,a_1\overline{a_2})  \right|    \ll_A \frac{\Psi(x,y)}{(\log x)^A}, \]
 for any integers $a_1,a_2$ for which $1\leq |a_1|, |a_2| \leq x^\delta$, with $\mathcal A=\mathcal A(D)$ where $D = (x/\Psi(x,y))^2 (\log x)^{2A+20}$.
 \end{theorem}
  
We prove this by modifying some of the estimates in   \cite{Dr1}. For any $D\geq 1$  and  integer  $q\geq 1$ let
\begin{equation}
\cu_D(n ; q) = \1_{n\equiv 1\mod{q}} - \frac1{\vphi(q)} \ssum{\chi\mod{q} \\ \cond(\chi) \leq D} \chi(n),\label{eq:def-cu}
\end{equation}
so that 
\[
\Delta_{\mathcal A}(f,x;q,a_1\overline{a_2}) = \sum_{n\leq x} f(n) \cu_D(n\overline{a_1}a_2 ; q).
\]
Note that~$\cu_D(n;q)=0$ unless~$(n, q)=1$ and~$q>D$, in which case
\begin{equation}  \label{eq:cu-trivial}
\begin{aligned}
|\cu_D(n; q)| {}& \leq \1_{n\equiv 1\mod{q}} + \frac1{\vphi(q)} \ssum{r\leq D \\ r|q} \varphi(r) \\
{}& \leq \1_{n\equiv 1\mod{q}} + \frac{D\tau(q)}{\vphi(q)}.
\end{aligned}
\end{equation}
For $(n,q)=1$, since
\[ \sum_{\substack{\chi\pmod q \\ \cond(\chi)\leq D}} \chi(n) = \sum_{\substack{s \leq D \\ s \mid q}} \sum_{\substack{\psi\pmod s \\ \psi\text{ primitive}}} \psi(n) = \sum_{\substack{s \leq D \\ s|q}} \sum_{d|s} \mu(s/d) \varphi(d) \1_{d | n-1}, \]
by letting $b = s/d$ we obtain the alternate expression
\begin{equation}  \label{eq:cu-moebius}
\cu_D(n; q) = \1_{n\equiv 1\mod{q}} - \frac1{\vphi(q)} \ssum{d\leq D \\ d| (q,n-1)} \varphi(d) \ssum{b\leq D/d \\ b| q/d}   \mu(b) .
\end{equation}
 Theorem~\ref{SaryResult}  is an immediate consequence of Theorem \ref{thm:equidist-noMT}.

\begin{theorem}\label{thm:equidist-noMT}
For any fixed~$\ee>0$, there exists~$C, \delta>0$ such that whenever
$$ 1 \leq D \leq x^\delta, \qquad (\log x)^C \leq y \leq x^\delta, $$
we have, uniformly for~$0<|a_1|, |a_2| \leq x^\delta$ and~$f\in{\mathcal C}$,
\begin{equation}
\ssum{q\leq x^{3/5-\ee} \\ (q, a_1a_2)=1} \Big| \sum_{n\in S(x, y)} f(n) \cu_D(n\overline{a_1}a_2 ; q) \Big| \ll_\ee D^{-\frac12} x (\log x)^{10}.\label{eq:bound-disc-cu}
\end{equation}
\end{theorem}

To prove Theorem~\ref{thm:equidist-noMT}, we first prove the following generalisation of Theorem 3 of~\cite{Dr1}, where the bound~$D$ on the conductor is allowed to vary.
\begin{lemma}\label{lemma:triconv}
Let~$M, N, L, R\geq 1$ and~$(\alpha_m)$, $(\beta_n)$, $(\lambda_\ell)$ be three sequences, bounded in modulus by~$1$, supported on integers inside~$(M, 2M]$, $(N, 2N]$, and~$(L, 2L]$ respectively. Let~$x = MNL$. For any fixed~$\ee>0$, there exists~$\delta>0$ such that whenever either the conditions (3.1), or the conditions (3.2) of~\cite{Dr1} are met, we have
\begin{equation}
\ssum{R<r\leq 2R \\ (r, a_1a_2)=1} \Big| \sum_m \sum_n \sum_\ell \alpha_m \beta_n \lambda_\ell \cu_D(mn\ell \overline{a_1} a_2; q) \Big| \ll D^{-\frac12}x(\log x)^3\label{eq:bound-triconv}
\end{equation}
for~$1\leq D \leq x^\delta$.
\end{lemma}

Theorem~3 of~\cite{Dr1} is the special case when~$D$ has maximal size,~$D=x^\delta$. The conditions (3.1) or (3.2) of~\cite{Dr1} concern the relative sizes of $M,N,L$. They are rather technical, but a critical case when the conditions are met is 
\[ R \approx x^{3/5}, \ \ M \approx x^{1/5}, \ \ N \approx x^{2/5}, \ \ L \approx x^{2/5}. \]

\begin{proof}
We follow closely the arguments of~\cite{Dr1}. Roughly speaking, the main point is that reducing the size of~$D$ only reduces the error terms, except in a certain diagonal contribution which yields the dominant error term, and which we analyse more carefully. Proceeding as in section 3 of~\cite{Dr1}, we reduce to the estimation of~${\mathcal S}_1 - 2\text{\rm Re} ({\mathcal S}_2) + {\mathcal S}_3$, where~${\mathcal S}_1$ is defined in the first display of~\cite[page 838]{Dr1},
$$ {\mathcal S}_2 = \ssum{R<r\leq 2R \\ (r, a_1a_2)=1} \frac1{\vphi(r)} \sum_{(m, r)=1} f(m) \summ_{\substack{(k_1, r)=1 \\ k_2 \equiv a_1\overline{a_2 m} \mod{r}}} u_{k_1}\overline{u_{k_2}} \ssum{\chi\mod{r} \\ \cond(\chi)\leq D}\chi(a_1\overline{a_2mk_2}), $$
and
$$ {\mathcal S}_3 = \ssum{R<r\leq 2R \\ (r, a_1a_2)=1} \frac1{\vphi(r)^2} \sum_{(m, r)=1} f(m) \summ_{(k_1k_2, r)=1 } u_{k_1}\overline{u_{k_2}} \summ_{\substack{\chi_1, \chi_2\mod{r} \\ \cond(\chi_j)\leq D}} \chi_1\overline{\chi_2}(a_1\overline{a_2m})\chi_1(k_1)\overline{\chi_2(k_2)}, $$
where $u_k = \sum_{k = n\ell} \beta_n\lambda_{\ell}$, and $f$ is a smooth function supported inside~$[M/2, 3M]$ satisfying~$\|f^{(j)}\|_\infty \ll_j M^{-j}$ for any $j \geq 0$.

The quantity~${\mathcal S}_1$ being the same as in~\cite{Dr1}, we can quote the estimate
\[ {\mathcal S}_1 = {\hat f}(0)X_1 + O(x^{1-\delta}KR^{-1}) \] 
from~\cite[formula~(3.17)]{Dr1}, where $K = NL$. Here~${\hat f}(0) = \int_\R f$, and~$X_1$ is defined at~\cite[formula~(3.12)]{Dr1}. For the estimation of~${\mathcal S}_2$ and~${\mathcal S}_3$, we reproduce sections~3.2 and~3.3 of~\cite{Dr1}, the only difference being that the set~${\mathcal X} = \{\chi\prim:\ \cond(\chi)\leq x^\ee\}$ is replaced with the subset~${\mathcal X} = \{\chi\prim:\ \cond(\chi)\leq D\}$. We claim that the estimates
\[ {\mathcal S}_j = {\hat f}(0)X_j + O(x^{1-\delta} KR^{-1}) \] 
hold for $j \in \{2,3\}$, with
$$ X_2 = \ssum{R<r\leq 2R \\ (r, a_1a_2)=1} \frac1{r\vphi(r)} \summ_{(k_1k_2, r)=1} \sum_{\substack{\chi \mod{r} \\ \cond(\chi)\leq D}} u_{k_1}\overline{u_{k_2}} \chi(k_1 \overline{k_2}), $$
$$ X_3 = \ssum{R<r\leq 2R \\ (r, a_1a_2)=1} \frac1r  \ssum{0<b\leq r \\ (b, r)=1} \Bigg| \frac1{\vphi(r)} \sum_{(k, r)=1} \ssum{\chi \mod{r} \\ \cond(\chi)\leq D} u_k\chi(k\overline{b}) \Bigg|^2. $$
To see this, we merely note that reducing the cardinality of~${\mathcal X}$, and the bound~$D$ on the conductors, leads to better error terms in the analysis. This is clear from the bound on~$R_3$ in~\cite[formula~(3.8)]{Dr1}, which grows proportionally to~$|{\mathcal X}|D^2$, and from the bound on~$\sum_{z\in{\mathcal Z}} \lambda(z)|G_2(z, \xi)|$ in~\cite[formula~(3.10)]{Dr1}, which grows proportionally to~$|{\mathcal X}|$.

Finally we are left with evaluating~$X_1 - 2\text{\rm Re}(X_2) + X_3$, which makes use of the multiplicative large sieve. Proceeding as in section~3.6 of~\cite{Dr1}, we find
\begin{align*}
{}& X_1 - 2\text{\rm Re} (X_2) + X_3 \\
\ll {}& R^{-1}(\log R)^2 \sum_{d\leq R} \frac{\tau(d)}{\vphi(d)} \int_D^\infty \Big(\min(2R, t)^2 + \frac Kd\Big)\sum_{K/d < k' \leq 2K/d} |u_{k'd}|^2 \frac{\dd t}{t^2} \\
\ll {}& R^{-1}(\log R)^2 K(\log K)^3 \sum_{d\leq R} \frac{\tau(d)^3}{d\vphi(d)} \Big(\frac K{dD} + R\Big) \\
\ll {}& (\log x)^5 K^2 (RD)^{-1}. \numberthis\label{bound-sumXj}
\end{align*}
Here we have used the bound\footnote{Note that there is a factor~$(\tau(d)\log K)^2$ missing in the third display, p.852 of~\cite{Dr1}.}~$|u_k|\leq \tau(k)$, and the hypothesis~$R\leq x^{-\ee}K \leq K/D$. Following~\cite[formula~(3.34)]{Dr1}, this leads to the upper bound
\begin{align*}
{}& \ssum{R<r\leq 2R \\ (r, a_1a_2)=1} \Big| \sum_m \sum_n \sum_\ell \alpha_m \beta_n \lambda_\ell \cu_D(mn\ell \overline{a_1} a_2; q) \Big| \\
\ll {}& (M^2R\{X_1 - 2\text{\rm Re} (X_2) + X_3\})^{1/2} + O(x^{1-\delta/3}).
\end{align*}
The claimed bound~\eqref{eq:bound-triconv} then follows by~\eqref{bound-sumXj}.
\end{proof}

To deduce Theorem~\ref{thm:equidist-noMT} from Lemma~\ref{lemma:triconv}, we start with the following special case of Theorem~\ref{thm:equidist-noMT}.

\begin{proposition}\label{prop:equidist-sqf}
Theorem~\ref{thm:equidist-noMT} holds true for functions~$f$ supported on squarefree integers.
\end{proposition}

\begin{proof}
We extend the arguments of pages~852-853 of~\cite{Dr1}, renaming the variable~$q$ into~$r$. Suppose first~$R\geq x^{4/9}$. We restrict~$n$ and~$r$ to dyadic intervals~$x<n\leq 2x$ and~$R<r\leq 2R$. Choosing the parameters~$(M_0, N_0, L_0)$ as in~\cite[p.852, last display]{Dr1}, we obtain
\begin{align*}
{}& \ssum{R<r\leq 2R \\ (r, a_1a_2)=1} \Big|\ssum{x<n\leq 2x \\ P^+(n)\leq y} f(n) \cu_D(n\overline{a_1}a_2;r)\Big| \\
= {}& \ssum{R<r\leq 2R \\ (r, a_1a_2)=1} \Big|\ssum{L_0<\ell\leq L_0P^-(\ell) \\  P^+(\ell)\leq y} \ssum{M_0<m\leq M_0 P^-(m) \\ P^+(m)\leq P^-(\ell)} \ssum{x<mn\ell \leq 2x \\ P^+(n) \leq P^-(m)} f(mn\ell) \cu_D(mn\ell\overline{a_1}a_2; r)\Big| \\
= {}& \ssum{R<r\leq 2R \\ (r, a_1a_2)=1} \Big|\ssum{L_0<\ell\leq L_0P^-(\ell) \\  P^+(\ell)\leq y} \ssum{M_0<m\leq M_0 P^-(m) \\ P^+(m)< P^-(\ell)} \ssum{x<mn\ell \leq 2x \\ P^+(n) < P^-(m)} f(m)f(n)f(\ell) \cu_D(mn\ell\overline{a_1}a_2; r)\Big|
\end{align*}
where we have used the fact that~$f$ is supported on squarefree integers in the last equality. The rest of the argument consists in cutting the sums over~$(m, n, \ell)$ in dyadic segments, and analytically separating the four conditions~$x<mn\ell\leq 2x$, $P^+(m)<P^-(\ell)$ and~$P^+(n)< P^-(m)$. The details are identical to the proof of Proposition~2 of~\cite{Dr1}, using our Lemma~\ref{lemma:triconv} instead of \cite[Theorem~3]{Dr1}; we obtain the bound
$$ O(D^{-\frac12} x (\log x)^7 (\log y)^3) = O(D^{-\frac12} x (\log x)^{10}). $$

The Bombieri-Vinogradov range~$R\leq x^{4/9}$ is covered by similar arguments, using~\cite[Theorem~17.4]{IwaniecKowalski} instead of Lemma~\ref{lemma:triconv}.
\end{proof}

\begin{proof} [Deduction of the full  Theorem~\ref{thm:equidist-noMT} from Proposition~\ref{prop:equidist-sqf}]   We  let ${\mathcal K}$ be the set of powerful numbers, that is for $k \in \mathcal K$ if prime $p$ divides $k$ then $p^2$ also divides $k$.
Note that $|{\mathcal K}\cap [1, x]|\ll x^{\frac12}$. Out of every~$n$ counted in the left-hand side of~\eqref{eq:bound-disc-cu}, we extract the largest powerful divisor~$k$. Then from the triangle inequality and the bound~$|f(k)|\leq 1$, the left-hand side of~\eqref{eq:bound-disc-cu} is at most
\begin{equation}
\ssum{q\leq x^{3/5-\ee} \\ (q, a_1a_2)=1} \ssum{k\in {\mathcal K}\cap  S(x, y) \\ (k, q)=1} \Big| \ssum{n\in S(x/k, y) \\ (n, k)=1} \mu^2(n)f(n)\cu_D(kna_2\overline{a_1}; q)\Big|.\label{eq:sqfl-extr}
\end{equation}
Let~$K\geq 1$ be a parameter. We use the trivial bound~\eqref{eq:cu-trivial} on the contribution of~$k>K$, getting
$$ \sum_{q\leq x^{3/5-\ee}} \ssum{k\in{\mathcal K} \\ k>K} \sum_{n\leq x/k} \Big(\1_{ka_2n\equiv a_1 \mod{q}} + \frac{D\tau(q)}{\vphi(q)}\Big) = T_1 + T_2, $$
say, where we have separated the contribution of the two summands. Executing the sum over~$q$ first, and separating the case~$kn|a_1$, we find
$$ T_1 \leq \sum_{q\leq x^{3/5-\ee}} \tau(|a_1|)^2 + \ssum{k\in{\mathcal K} \\ k>K} \sum_{n\leq x/k} \tau(|kna_2 - a_1|) \ll x^{4/5} + x^{1+\ee} D K^{-\frac12}. $$
It is easy to see that~$T_2 \ll x^{1+\ee} D K^{-\frac12}$ as well. Next, to each~$1\leq k \leq K$ in~\eqref{eq:sqfl-extr}, by hypothesis, we may use Proposition~\ref{prop:equidist-sqf} with~$x\gets x/k$, $a_2 \gets ka_2$ and~$f(n) \gets \1_{(n, k)=1} \mu^2(n)f(n)$, and obtain the existence of~$C, \delta_1>0$ such that, for~$|a_1|,|a_2k| \leq x^{\delta_1}$ and~$(\log x)^C \leq y \leq x^{\delta_1}$,
$$ \ssum{q\leq x^{3/5-\ee} \\ (q, a_1a_2)=1} \Big| \ssum{n\in S(x/k, y) \\ (n, k)=1} \mu^2(n)f(n)\cu_D(n\overline{a_1}ka_2; q)\Big| \ll D^{-\frac12} k^{-1} x (\log x)^{10}. $$
We take~$\delta = \delta_1/7$, $K=x^{\delta_1/2}$, and sum over~$k\leq K$, using~$\sum_{k\in{\mathcal K}} k^{-1} < \infty$. By hypothesis~$D\leq x^\delta$, so that~$x^{4/5}\ll DK^{-\frac12}x \ll D^{-\frac12} x^{1-\delta/2}$, and we find that~\eqref{eq:sqfl-extr} is at most~$\ll D^{-\frac12} x (\log x)^{10}$ as claimed.
\end{proof}

 \section{Altering the set of exceptional characters}

To prove Theorem \ref{MainTheorem} we need to reduce the set of exceptional characters from $\mathcal A(D)$ to $\Xi$. We shall set this up in Proposition \ref{Moduli Reduction}.
  
It is convenient to write $b=a_1/a_2$ (which is $\equiv a_1\overline{a_2} \pmod q$) and to define $(q,b)$ to mean $(q,a_1a_2)$. Thus in Theorem \ref{SaryResult}  we are working with
\[
\sum_{\substack{q\leq Q \\ (q,b)=1}}   \left|  \Delta_{\mathcal A}(f,x;q,b)  \right|  
\]
for $Q=x^{3/5-\varepsilon}$.

\begin{lemma} \label{Lemm X and A} Let $\mathcal A = \mathcal A(D)$ for some $D \geq 2$. Suppose that $\Xi\subset   \mathcal A$.  If $(b,q)=1$ then
\[
 \begin{split}
& \Delta_\Xi(f,x;q,b) -  \Delta_{\mathcal A}(f,x;q,b)  \\ 
  &=   \frac 1{\varphi(q)} \sum_{\substack{\ell\geq 1 \\ p|\ell \implies p|q}} g(\ell)   
    \sum_{\substack{ d\leq D \\  (d,\ell)=1 \\ d|q}} \varphi(d) \Delta_{\Xi}(f,x/\ell;d,b\overline{\ell})
  \sum_{\substack{n\leq D/d \\ n|q/d}}  \mu(n),
 \end{split}
 \]
where $g$ is the multiplicative function with $F(s)G(s)=1$.
\end{lemma}

\begin{proof}
 If $(b,q)=1$ then
 \[
 \begin{split}
\Delta_\Xi(f,x;q,b) -  \Delta_{\mathcal A}(f,x;q,b) &= 
   \frac 1{\varphi(q)} \sum_{ \substack{  \chi  \in   {\mathcal A}_q \\  \chi  \not\in   \Xi_q  }}  \chi(b) S_f(x,\chi) \\ 
  &=   \frac 1{\varphi(q)}  \sum_{\substack{m\leq D \\ m|q}} \sum_{\substack{  \chi\in \mathcal P(m)_q\\  \chi  \not\in   \Xi_q  }} \chi(b) S_f(x,\chi) ,
\end{split}
 \]
where $\mathcal P(m)$ denotes the set of primitive characters $\pmod m$, as $\mathcal A$ is the set of all primitive characters of conductor $\leq D$ . 
Let
$\mathcal C(m)$ denote  the  set of all characters $\pmod m$. For $m|q$ we define
\[
 \begin{split}
\Delta_{\Xi,q}(f,x;m,b) &:=\sum_{\substack{n\leq x \\ n\equiv a \pmod m \\ (n,q)=1}} f(n) -   
\frac 1{\varphi(m)} \sum_{ \substack{  \chi  \in  \mathcal C(m)_q \cap \Xi_q}} \chi(b) S_f(x,\chi)\\
&= \frac 1{\varphi(m)} \sum_{\substack{  \chi\in \mathcal C(m)_q\\  \chi  \not\in   \Xi_q  }} \chi(b) S_f(x,\chi) 
\\
&= \frac 1{\varphi(m)} \sum_{d|m}  \sum_{\substack{  \chi\in \mathcal P(d)_q\\  \chi  \not\in   \Xi_q  }} \chi(b) S_f(x,\chi) .
\end{split}
\]
By M\"obius inversion we deduce that, for $m|q$,
\[
 \sum_{\substack{  \chi\in \mathcal P(m)_q\\  \chi  \not\in   \Xi_q  }} \chi(b) S_f(x,\chi)  = \sum_{d|m} \mu(m/d) \varphi(d) \Delta_{\Xi,q}(f,x;d,b) .
\]   

Next we wish to better understand $\Delta_{\Xi,q}(f,x;m,a)$. Let $f_q(p^k)=f(p^k)$ if $p|q,\ p\nmid m$, and $f_q(p^k)=0$ otherwise. Define $g_q$ from $g$ in a similarly way. Note $f_q$ and $g_q$ are simply $f$ and $g$ supported on the integers composed from the prime factors of $q$.
If $(a,m)=1$ then 
\[
\Delta_{\Xi}(f,x;m,a) = \sum_{\substack{\ell\geq 1 \\ (\ell,m)=1}} f_q(\ell)  \Delta_{\Xi,q}(f,x/\ell;m,a\overline{\ell}),
\]
and since $F_qG_q = 1$ we have
\begin{align*}
\Delta_{\Xi,q}(f,x;m,a) &= \sum_{\substack{\ell\geq 1 \\ (\ell,m)=1}} g_q(\ell)  \Delta_{\Xi}(f,x/\ell;m,a\overline{\ell}) \\
&= \sum_{\substack{\ell\geq 1 \\ p|\ell \implies p|q \\ (\ell,m)=1}} g(\ell)  \Delta_{\Xi}(f,x/\ell;m,a\overline{\ell}).
\end{align*}
Substituting this in above then yields
\[
 \sum_{\substack{  \chi\in \mathcal P(m)_q\\  \chi  \not\in   \Xi_q  }} \chi(b) S_f(x,\chi)  = \sum_{d|m} \mu(m/d) \varphi(d)
 \sum_{\substack{\ell\geq 1 \\ p|\ell \implies p|q,\\ (\ell,d)=1}} g(\ell)  \Delta_{\Xi}(f,x/\ell;d,b\overline{\ell}),
  \]   
and the result follows writing  $m=dn$.
\end{proof}

\begin{proposition} \label{Moduli Reduction} Let the notations and assumptions be as in the statement of Theorem~\ref{SaryResult}. Suppose that $\Xi\subset   \mathcal A$. Then
\[
 \begin{split}
&  \sum_{\substack{q\leq x^{3/5-\ee} \\ (b,q)=1}}   \left|  \Delta_\Xi(f,x;q,b)  \right|  \leq 
 O\left(  \frac{\Psi(x,y)}{(\log x)^A} \right)  \\ 
  &+   (\log x)^2 \sum_{\substack{\ell\leq X \\ L:=\prod_{p|\ell} p}}  \frac {\tau(L)}{ \varphi(L)}  
    \sum_{\substack{ d\leq D \\  (d,\ell)=1}}   | \Delta_{\Xi}(f,x/\ell;d,b\overline{\ell}) |,
 \end{split}
 \]
 where $X=(1+\beta)^2 (\log x)^{2A+6}$ with 
 \[
 \beta=\beta(\Xi) := \sum_{\psi \pmod{r_\psi} \in \Xi} \frac 1{r_\psi} . 
 \]
\end{proposition}

\begin{proof}
Set $Q = x^{3/5-\ee}$. We deduce from Lemma \ref{Lemm X and A}, as  each $|g(\ell)|\leq 1$ since $f,g\in \mathcal C$,  that
\[
 \begin{split}
&  \sum_{\substack{q\leq Q \\ (b,q)=1}}   \left|  \Delta_\Xi(f,x;q,b)  \right|  \leq 
 \sum_{\substack{q \leq Q \\ (b,q)=1}}  \left|  \Delta_{\mathcal A}(f,x;q,b)  \right|   \\ 
  &+  \sum_{\substack{\ell\geq 1 \\ L:=\prod_{p|\ell} p}}  
    \sum_{\substack{ d\leq D \\  (d,\ell)=1}} \varphi(d) | \Delta_{\Xi}(f,x/\ell;d,b\overline{\ell}) |
   \sum_{\substack{q \leq Q \\ (b,q)=1 \\  dL|q }}   \frac 1{\varphi(q)}
     \left|  \sum_{\substack{n\leq D/d \\ n|q/d}}  \mu(n)   \right|  .
 \end{split}
 \]
 The first term on the right-hand side is $ \ll_A \Psi(x,y)/(\log x)^A$  by Theorem \ref{SaryResult}.
 For the sum at the end we have an upper bound
\[
\leq  \sum_{\substack{q \leq Q   \\  dL|q }}   \frac {\tau^*(q/d)}{\varphi(q)} 
\leq  \frac {\tau^*(L)}{\varphi(dL)}  \sum_{\substack{r  \leq Q /dL   }}   \frac {\tau^*(r)}{\varphi(r)} \ll  \frac {\tau(L)}{\varphi(d)\varphi(L)}  (\log x)^2,
\]
where $\tau^*(m)$ denotes the number of squarefree divisors of $m$, writing $q=dLr$, as $L$ is squarefree.  Therefore
\[
 \begin{split}
&  \sum_{\substack{q\leq Q \\ (b,q)=1}}   \left|  \Delta_\Xi(f,x;q,b)  \right|  \leq 
O  \left( \frac{\Psi(x,y)}{(\log x)^A}  \right)  \\ 
  &+  (\log x)^2 \sum_{\substack{\ell\geq 1 \\ L:=\prod_{p|\ell} p}}    \frac {\tau(L)}{ \varphi(L)}  
    \sum_{\substack{ d\leq D \\  (d,\ell)=1}}   | \Delta_{\Xi}(f,x/\ell;d,b\overline{\ell}) |.
 \end{split}
 \]

We will attack this last sum first by employing relatively  trivial bounds for the terms with $\ell$ that are not too small, so that we only have to consider $\ell d$ that are smallish in further detail. Now Theorem 1 of \cite{FT} gives the upper bound
\begin{equation} \label{SmoothUB}
 \Psi_q(x,y) \ll  \frac  {\varphi(q)} q \Psi (x,y)  
\end{equation}
provided $x\geq y\geq \exp((\log\log x)^2)$ and $q\leq x$. Therefore
\[
|\Delta_\Xi(f,x;q,a)| \leq   \Psi(x,y;a,q) +  \frac {| \Xi_q|} {\varphi(q)}  \Psi_q(x,y)  \ll   \Psi(x,y;a,q) +   \frac {| \Xi_q|} {q}  \Psi(x,y) 
 \]
 in this range. Substituting in, the upper bound on the $\ell$th term above becomes
 \[
\ll ( \log x)^2 \cdot   \frac {\tau(L)}{ \varphi(L)} \left( 
 \sum_{\substack{ d\leq D \\  (d,\ell)=1}}  \Psi(x/\ell,y;d,b\overline{\ell})  +
 \Psi(x/\ell,y)      \sum_{\substack{ d\leq D \\  (d,\ell)=1}}  \frac {| \Xi_d|} {d}  \right).
 \]
The second sum over $d$ is therefore 
\[
 \sum_{\psi \pmod {r_\psi}\in \Xi}   \sum_{\substack{ d\leq D \\  (d,\ell)=1 \\ r_\psi|d}}  \frac1d
 \ll    \beta    \frac{\varphi(\ell)}{\ell} \log D  .
\]
For the first term we use Theorem 1 of \cite{Har} which yields that
\[
\sum_{\substack{ d\leq D \\  (d,\ell)=1}} \left|  \Psi(x/\ell,y;d,b\overline{\ell}) - \frac{ \Psi_d(x/\ell,y) } {\varphi(d)} \right| \ll \frac{\Psi(x/\ell,y)}{(\log x)^A},
\]
as $D\leq \sqrt{\Psi(x,y)}/(\log x)^B$ since $y\geq (\log x)^C$.
Therefore, expanding the sum and using \eqref{SmoothUB}, we obtain
\[
\sum_{\substack{ d\leq D \\  (d,\ell)=1}}    \Psi(x/\ell,y;d,b\overline{\ell}) \ll 
\sum_{\substack{ d\leq D \\  (d,\ell)=1}}  \frac{ \Psi(x/\ell,y) } d  +\frac{\Psi(x/\ell,y)}{(\log x)^A}\ll
 \frac{\varphi(\ell)}{\ell} \log D  \cdot \Psi(x/\ell,y).
\]

By Th\'eor\`eme 2.1 of~\cite{dBT} we have
\begin{equation} \label{Psi growth}
 \Psi(x/\ell,y)  \ll \Psi(x,y)/\ell^\alpha,
\end{equation}
where $\alpha >3/4$ in our range for $y$.
Therefore in total the $\ell$th term in 
 \[
\ll   \frac {\tau(L)}{ L\ell^\alpha}  \Psi(x ,y)   (\log x)^2(\log D) \cdot ( 1+ \beta)    .
 \]
Summing over $\ell>X$, we  obtain, taking $\sigma=\alpha-1/4$,
\[
    \sum_{\substack{\ell\geq X \\ L:=\prod_{p|\ell} p}}  \frac {\tau(L)}{ L \ell^{\alpha} }  \leq \sum_{\substack{\ell\\ L:=\prod_{p|\ell} p}}  \frac {\tau(L)}{ L \ell^{\alpha} } (\ell/X)^{\sigma} =
    X^{-\sigma} \prod_p \left( 1 + \frac 2{p(p^{1/4}-1 )} \right) \ll X^{-1/2}.
\]
Taking $X=(1+\beta)^2 (\log x)^{2A+6}$,   the contribution of the    $\ell>X$ is therefore
$ \ll \Psi(x ,y)/(\log x)^A$.
 Combining all of the above then yields the result. 
  \end{proof}

  \section{Putting the pieces together}

In order to prove   Theorem~\ref{MainTheorem} we need Theorem~\ref{SaryResult},
Proposition \ref{Moduli Reduction}, Corollary~\ref{cor:exceptional-good-bound} (which will be proved in the final two sections), and  the following result which is Proposition 5.1 of \cite{GSh}.
 
 \begin{proposition} \label{Errorfaps1} Fix $B\geq 0$ and $0<\eta<\frac 12$.  Given $ (\log x)^{4B+5}\leq y=x^{1/u}\leq x^{1/2-\eta}$ let 
\[
R=R(x,y):=\min\{ y^{ \frac{  \log\log\log x}{3\log u}},  x^{\frac \eta{3\log\log x}}\} \ (\leq y^{1/3}).
\]   
Suppose that $f\in \mathcal C$, and is only supported on $y$-smooth numbers. 
There exists a set, $\Xi$, of primitive characters $\psi \pmod r$ with $r\leq R$,  such that  if $q\leq R$ and $(a,q)=1$ then
\[    
| \Delta_{\Xi}(f,x;q,a) |  \ll  \frac 1{\varphi(q)}  \frac{\Psi(x,y)} {(\log x)^{B}} .
\]
Moreover, one may take $\Xi$ to be $\Xi=\Xi(2B+2\frac13)$, where    $\Xi(C)$ is the set of primitive characters $\psi \pmod r$ with $r\leq R$ such that there exists $x^{\eta}< X\leq x$ for which
\begin{equation} \label{Hyp3}
|S_f(X,\psi)|    \geq \frac {\Psi(X,y)}{(u\log u)^4(\log x)^{C}}  .
\end{equation}
\end{proposition}

\begin{proof} [Proof of Theorem~\ref{MainTheorem}] Let $c>0$ be the small constant from Corollary~\ref{cor:exceptional-good-bound}. Set $D = (x/\Psi(x,y))^2 (\log x)^{2A+20}$, and one easily verifies that the hypothesis for $y$ implies that 
\begin{equation}\label{eq:Dupperbound} 
D \leq \min(R, y^c, \exp(c\log x/\log\log x)) 
\end{equation}
from the usual estimate $\Psi(x,y)=xu^{-u+o(u)}$ for smooth numbers.
We will prove Theorem~\ref{MainTheorem} with $\Xi = \Xi(2A+8\frac13)$, where $\Xi(C)$ is the set of primitive characters $\psi\pmod{r}$ with $r \leq D$, such that there exists $x^{1/4} < X \leq x$ for which
\eqref{Hyp3} holds.
By Proposition~\ref{Errorfaps1} with $B = A+3$ and $\eta = 1/4$, we have the bound
\[ |\Delta_{\Xi}(f,x;q,a)| \ll \frac{1}{\varphi(q)} \frac{\Psi(x,y)}{(\log x)^{A+3}} \]
whenever $q \leq D$ and $(a,q)=1$. Moreover, we have the same bound with $x$ replaced by $x/\ell$ for any $\ell = x^{o(1)}$.

The goal of the next two sections will be to prove Corollary~\ref{cor:exceptional-good-bound}, which implies that 
\[ |\Xi| \ll (\log x)^{6A+38}. \]
This implies that $\beta(\Xi)\ll (\log x)^{3A+19}$, and so 
$X\ll (\log x)^{8A+44}$ in Proposition \ref{Moduli Reduction}. Thus for each $\ell \leq X$ and $d \leq D$, we have
  \[
   | \Delta_{\Xi}(f,x/\ell;d,b\overline{\ell}) |  \ll  \frac 1{\varphi(d)}  \frac{\Psi(x/\ell,y)} {(\log x)^{A+3}}  \ll  \frac 1{\varphi(d) \ell^\alpha}  \frac{\Psi(x ,y)} {(\log x)^{A+3}}  ,
   \]
by \eqref{Psi growth}. Therefore 
\[ 
 \begin{split}
&    (\log x)^2 \sum_{\substack{\ell\leq X \\ L:=\prod_{p|\ell} p}}  \frac {\tau(L)}{ \varphi(L)}  
    \sum_{\substack{ d\leq D \\  (d,\ell)=1}}   | \Delta_{\Xi}(f,x/\ell;d,b\overline{\ell}) | \\
& \ll  \frac{\Psi(x ,y)} {(\log x)^{A+1}} \sum_{\substack{\ell\leq X \\ L:=\prod_{p|\ell} p}}  \frac {\tau(L)}{ \varphi(L)\ell^\alpha }  
    \sum_{\substack{ d\leq D \\  (d,L)=1}}  \frac 1{\varphi(d) }  
     \ll  \frac{\Psi(x ,y)} {(\log x)^{A }} \sum_{\substack{\ell\leq X \\ L:=\prod_{p|\ell} p}}  \frac {\tau(L)}{ L\ell^\alpha }   \\
   &\leq \frac{\Psi(x ,y)} {(\log x)^{A }} \prod_{p\leq X}  \left( 1 + \frac {2}{ p(p^\alpha -1)}   \right) \ll \frac{\Psi(x ,y)} {(\log x)^{A }}  .
 \end{split}
 \]
We therefore deduce from Proposition \ref{Moduli Reduction}   that  
\[
   \sum_{\substack{q\leq x^{3/5-\varepsilon} \\ (q,a_1a_2)=1}}   \left|  \Delta_\Xi(f,x;q,a_1\overline{a_2})  \right|
    \ll 
   \frac{\Psi(x,y)}{(\log x)^A}  
    \]
as desired.

To deduce the second part of Theorem~\ref{MainTheorem}, about functions $f$ satisfying the Siegel-Walfisz criterion, we use the following variant of Proposition 3.4 in \cite{GSh}:
  
  \begin{proposition} \label{Using SW}  Fix $\ee > 0$. Let $(\log x)^{1+\ee} \leq y \leq x$ be large. Let $f\in \mathcal C$ be a multiplicative function supported on $y$-smooth integers. Suppose that $\Xi$ is a set of primitive characters, containing $\ll   (\log x)^{C}$ elements,
such that 
 \[
 \sum_{q \sim Q}   \left|  \Delta_{\Xi}(f,x;q,a_q)  \right|  \ll \frac {\Psi(x,y)}{(\log x)^B},
 \]
 for $(a_q,q)=1$ for all $q\sim Q$, where $Q\leq x$.
If the $D$-Siegel-Walfisz criterion holds for $f$, where $D\geq B+C$, then  
\[ 
\sum_{q \sim Q} \left|  \Delta(f,x;q,a_q)  \right|    \ll \frac {\Psi(x,y)}{(\log x)^B}.
\]
\end{proposition}

\begin{proof}
By the definition of $\Delta_{\Xi}$ we have
\[ \left|  \Delta(f,x;q,a_q)  \right|  \leq \left|  \Delta_{\Xi}(f,x;q,a_q)  \right|  + \frac{1}{\varphi(q)} \sum_{\substack{\chi\pmod q \\ \chi \in \Xi_q, \chi\neq\chi_0}} |S_f(x,\chi)|. \]
Summing this over $q \sim Q$ and using the hypothesis, we deduce that
\[ \sum_{q \sim Q} \left|  \Delta(f,x;q,a_q)  \right|  \leq \sum_{\substack{\psi \in \Xi \\ \psi\neq 1}} \sum_{\substack{r_{\psi} \mid q \sim Q \\ \chi\pmod{q}\text{ induced by }\psi}} \frac{|S_f(x,\chi)|}{\varphi(q)} + O\left(\frac{\Psi(x,y)}{(\log x)^B}\right).  \]
It suffices to show that, for each fixed $\psi\pmod{r} \in \Xi$ with $r>1$, we have
\begin{equation}\label{eq:UsingSW0} 
\sum_{\substack{r \mid q \sim Q \\ \chi\pmod{q}\text{ induced by }\psi}} \frac{|S_f(x,\chi)|}{\varphi(q)} \ll \frac{\Psi(x,y)}{(\log x)^{D}}. 
\end{equation}
The conclusion then follows since $|\Xi| \ll (\log x)^C$ and $D \geq B+C$. If $\chi\pmod{q}$ is induced by $\psi\pmod{r}$, then there is a multiplicative function $h$ supported only on powers of primes which divide $q$ but not $r$, such that $h * f\overline{\psi} = f\overline{\chi}$. Note that $h \in \mathcal C$ since $f \in \mathcal C$, and in particular $h$ is $1$-bounded. It follows that
\[ |S_f(x,\chi)| = \left|\sum_{m \leq x} h(m) S_f(x/m,\psi)\right| \leq \sum_{\substack{m \leq x \\ p\mid m \Rightarrow p\mid q, p\nmid r}} |S_f(x/m,\psi)|. \]
Since $f$ satisfies the $D$-Siegel-Walfisz criterion, we have
\[ \frac{S_f(x/m,\psi)}{\varphi(r)} = \frac{1}{\varphi(r)} \sum_{a\pmod r} \overline{\psi}(a) \Delta(f, x/m; r,a) \ll \frac{\Psi(x/m,y)}{(\log (x/m))^D}. \]
Using the bound $\Psi(x/m,y) \ll m^{-\alpha}\Psi(x,y)$ where $\alpha = \alpha(x,y) \geq \ee + o(1)$, we may bound the left hand side of~\eqref{eq:UsingSW0} by
\begin{equation}\label{eq:UsingSW1}
\Psi(x,y) \sum_{r\mid q \sim Q} \frac{\varphi(r)}{\varphi(q)}  \sum_{\substack{m \leq x \\ p\mid m\Rightarrow p\mid q,p\nmid r}} \frac{1}{m^{\alpha}(\log (x/m))^D}.  
\end{equation}
To analyze the inner sum over $m$, we break it into two pieces depending on whether $m \leq x^{1/2}$ or $m > x^{1/2}$:
\[ \sum_{\substack{m \leq x \\ p\mid m\Rightarrow p\mid q,p\nmid r}} \frac{1}{m^{\alpha}(\log (x/m))^D} \ll \frac{1}{(\log x)^D} \sum_{\substack{m \leq x^{1/2} \\ p\mid m\Rightarrow p\mid q,p\nmid r}} \frac{1}{m^{\alpha}} + \sum_{\substack{x^{1/2} < m \leq x \\ p\mid m\Rightarrow p\mid q,p\nmid r}} \frac{1}{m^{\alpha}}. \]
To deal with the sum over $x^{1/2} < m \leq x$, note that the  number of $m \leq x$ with $p\vert m \Rightarrow p\vert q$ is maximized when $q$ is the product of primes up to $\sim \log x$, in which case the number of such $m$ is $x^{o(1)}$. Thus the sum over $x^{1/2} < m \leq x$ is $O(x^{-\alpha/2+o(1)})$, and their overall contribution to~\eqref{eq:UsingSW1} is acceptable. To deal with the sum over $m \leq x^{1/2}$, note that
\[ \sum_{\substack{m \leq x^{1/2} \\ p\mid m\Rightarrow p\mid q,p\nmid r}} \frac{1}{m^{\alpha}} \leq \sum_{\substack{m \leq x^{1/2} \\ m\mid q, (m,r)=1}} \frac{\mu^2(m)}{\varphi_{\alpha}(m)}, \]
where $\varphi_{\alpha}(m) = \prod_{p \mid m}(p^{\alpha}-1)$. Thus their overall contribution to~\eqref{eq:UsingSW1} is
\[ \ll \frac{\Psi(x,y)}{(\log x)^D} \sum_{(m,r)=1} \frac{\mu^2(m)}{\varphi_{\alpha}(m)}  \sum_{mr\mid q\sim Q} \frac{\varphi(r)}{\varphi(q)} \ll \frac{\Psi(x,y)}{(\log x)^D} \sum_m \frac{\mu^2(m)}{\varphi(m) \varphi_{\alpha}(m)}  \ll \frac{\Psi(x,y)}{(\log x)^D}, \]
since the infinite sum over $m$ converges. This establishes~\eqref{eq:UsingSW0} and completes the proof of the lemma.
\end{proof}

To deduce the second part of Theorem~\ref{MainTheorem}, we apply Proposition~\ref{Using SW} with $a_q\equiv a_1\overline{a_2} \pmod q$   for each dyadic interval with $Q\leq x^{3/5-\varepsilon}$. It is applicable since
 we assume that the Siegel-Walfisz criterion holds for $f$ with exponent $D\geq A +(6A+38)$.
 Summing up over all dyadic intervals, 
 the second part of the Theorem follows (with $A$ replaced by $A-1$).
\end{proof}

\begin{remark}\label{rem:y-range}
The lower bound required for $y$ in the hypothesis~\eqref{y-range} comes precisely from~\eqref{eq:Dupperbound}. In fact,  one can still deduce Theorem~\ref{MainTheorem} even if we just had
\[ D \leq \min(R, x^c) \]
instead. To see this, we follow the arguments above but now use Corollary~\ref{cor:exceptional-good-bound-variant} to get
\[ \sum_{\psi\pmod{q} \in \Xi} \frac{1}{q^{1/2}} \ll (\log x)^{6A+38}. \]
We cannot directly apply Proposition~\ref{Using SW} now, but if we let $\mathcal E$ be the set of $\psi\pmod{r_{\psi}} \in \Xi$ with $r_{\psi} > (\log x)^{14A+78}$, then removing $\mathcal E$ from $\Xi$ induces an error of at most
\[ \sum_{q \leq x} \frac{1}{\varphi(q)} \sum_{\chi \in \mathcal E_q} |S_f(x,\chi)|. \]
As $|S_f(x,\chi) |\ll \Psi_q(x,y) \ll (\varphi(q)/q) \Psi(x,y)$   the above is
\[
\ll  \Psi(x,y)  \sum_{ \psi  \in   \mathcal E}   \sum_{\substack{q\leq x \\ r_\psi|q}}  \frac 1{q}  \ll  \Psi(x,y)\log x  \sum_{ \psi  \in   \mathcal E}  \frac 1{r_\psi} \ll \frac{\Psi(x,y)}{(\log x)^A},
\]
since
\[
\sum_{ \psi  \in   \mathcal E}  \frac 1{r_\psi}  \leq (\log x)^{-7A-39} \sum_{\psi \in \Xi} \frac{1}{r_{\psi}^{1/2}}  \ll (\log x)^{-A-1}.
\]
Thus $|\Xi \setminus \mathcal E| \leq (\log x)^{28A + 156}$, and  we may apply Proposition~\ref{Using SW} with $\Xi$ replaced by $\Xi \setminus \mathcal E$ to conclude the proof.

Thus if one is able to prove Proposition~\ref{Errorfaps1} with $R = x^c$  in the full range 
\begin{equation}\label{y-range-hyp}
x^{\delta} > y \geq (\log x)^C,
\end{equation}
then one can deduce Theorem~\ref{MainTheorem} with~\eqref{y-range} replaced by~\eqref{y-range-hyp}, for some $C$ sufficiently large in terms of $c$. 
\end{remark}

\section{A large sieve inequality supported on smooth numbers}

In this section we prove a large sieve inequality for sequences supported on smooth numbers, a result which may be of independent interest.

\begin{theorem}[Large sieve for smooth numbers]\label{thm:large-sieve-smooth}
There exists $C,c>0$ such that the following statement holds. Let $(\log x)^{C} \leq y \leq x$ be large, and
\[ Q = \min(y^c, \exp(c\log x/\log\log x)). \]
For any sequence $\{a_n\}$ we have
\[ \sum_{q \leq Q} \sum^*_{\chi\pmod q} \left| \sum_{\substack{n \leq x \\ P(n) \leq y}} a_n \chi(n) \right|^2 \ll \Psi(x,y) \cdot \sum_{\substack{n \leq x \\ P(n) \leq y}} |a_n|^2. \]
\end{theorem}

The upper bound is sharp up to a constant, as may be seen by taking each $a_n=1$, so that the $\chi=1$ term on the left-hand side equals $\Psi(x,y)^2$, the size of the right-hand side. This result has the advantage over the traditional large sieve inequality  that the sequence $\{a_n\}$ is supported on a sparse set (when $y = x^{o(1)}$), but the disadvantage that 
this inequality holds in a much smaller range for $q$ than the usual   $q\ll x^{1/2}$. It may well be that Theorem \ref{thm:large-sieve-smooth}  holds with $Q = \Psi(x,y)^{1/2}$.

\subsection{Zero-density estimates}

To prove Theorem~\ref{thm:large-sieve-smooth}, we will use the following two  consequences of deep zero-density results in the literature. The first is a bound for character sums over smooth numbers assuming a suitable zero-free region for the associated $L$-function (see Section 3 of~\cite{Har}).

\begin{proposition}\label{prop:zerofree-to-bound}
There is a small positive constant $\delta > 0$ and a large positive constant $\kappa > 0$ such that the following statement holds. Let $(\log x)^{1.1} \leq y \leq x$ be large. Let $\chi\pmod q$ be a non-principal character with $q \leq x$ and conductor $r := \operatorname{cond}(\chi) \leq x^\delta$. If $L(s,\chi)$ has no zeros in the region
\begin{equation}\label{0freeregion}
\text{\rm Re}(s) > 1-\varepsilon, \ \ |\text{\rm Im}(s)| \leq T,  
\end{equation} 
where the parameters $\ee,T$ satisfy
\begin{equation}\label{eq:hyp-epsilon-H}
\frac \kappa{\log y} < \varepsilon \leq \frac{\alpha(x,y)}2, \ \ y^{0.9\varepsilon} (\log x)^2 \leq T \leq x^\delta,
\end{equation} 
and moreover
\begin{equation}\label{eq:hyp-epsilon-H-2} 
\emph{either}\ \  y \geq (Tr)^\kappa \ \ \emph{or}\ \ \ee \geq 40\log\log (qyT)/\log y. 
\end{equation}
Then  
\[ \left|\sum_{\substack{n \leq x \\ P(n) \leq y}} \chi(n) \right| \ll \Psi(x,y) \sqrt{(\log x)(\log y)} (x^{-0.3\varepsilon} \log T + T^{-0.02}). \]
\end{proposition}  
  
We also need the following log-free zero-density estimate by Huxley and Jutila, which can be found in Section 2 of~\cite{Har}:

\begin{proposition}\label{prop:log-zerofree}
Let $\varepsilon \in [0,1/2]$, $T \geq 1$ and $Q \geq 1$. Then the function $G_Q(s) = \prod_{q\leq Q} \prod^*_{\chi\pmod q} L(s,\chi)$ has $\ll (Q^2T)^{\frac 52\varepsilon}$ zeros $s$, counted with multiplicity, inside the region \eqref{0freeregion}.
\end{proposition}

\subsection{Proof of Theorem~\ref{thm:large-sieve-smooth}}
  
It suffices to prove its dual form:

\begin{proposition}
There exist $C,c>0$ such that the following statement holds. Let $(\log x)^{C} \leq y \leq x$ be large, and
\[ Q = \min(y^c, \exp(c\log x/\log\log x)). \] 
For any sequence $\{b_{\chi}\}$ we have
\[ \sum_{\substack{n \leq x \\ P(n) \leq y}} \left| \sum_{q \leq Q} \sum^*_{\chi\pmod q} b_{\chi} \chi(n) \right|^2 \ll \Psi(x,y) \cdot \sum_{q \leq Q} \sum^*_{\chi \pmod q} |b_{\chi}|^2. \]
\end{proposition}

\begin{proof}
The left hand side can be bounded by
\begin{equation} \label{eq: LHS}
 \sum_{\chi_1} \sum_{\chi_2} b_{\chi_1} \overline{b_{\chi_2}} \sum_{\substack{n \leq x \\ P(n) \leq y}} (\chi_1\overline{\chi_2})(n). 
 \end{equation}
Thus we need to understand character sums over smooth numbers. The contribution from the diagonal terms with $\chi_1 = \chi_2$ is clearly acceptable, and thus we focus on non-diagonal terms. For $\eta \in (0,1/2]$, define $\Xi(\eta)$ to be the set of all non-principal characters $\chi\pmod q$ with $q \leq Q^2$, such that
\[ \eta \Psi(x,y) < \left|\sum_{\substack{n \leq x \\ P(n) \leq y}} \chi(n) \right| \leq 2\eta \Psi(x,y). \]
Furthermore, define $\Xi^*(\eta)$ to be the set of primitive characters which induce a character in $\Xi(\eta)$. The contribution to \eqref{eq: LHS} from those $\chi_1,\chi_2$ with $\chi_1\overline{\chi_2} \in \Xi(\eta)$ is  
\begin{align*}
& \leq  2 \eta \Psi(x,y) \sum_{\substack{\chi_1,\chi_2 \\ \chi_1\overline{\chi_2} \in \Xi(\eta)}} |b_{\chi_1}b_{\chi_2}| \leq 4\eta \Psi(x,y) \sum_{\substack{\chi_1,\chi_2 \\ \chi_1\overline{\chi_2} \in \Xi(\eta)}} |b_{\chi_1}|^2 \\
& \leq  4\eta \Psi(x,y) \sum_{\psi \in \Xi^*(\eta)} \sum_{\chi_1} |b_{\chi_1}|^2 \sum_{\substack{\chi_2 \\ \chi_1\overline{\chi_2}\text{ induced by }\psi}} 1.
\end{align*}
We claim that for a given $\chi_1\pmod{q_1}$ and a given primitive $\psi \in \Xi^*(\eta)$, there is at most one primitive character $\chi_2$ such that $\chi_1\overline{\chi_2}$ is induced by $\psi$. To see this, suppose that there are two primitive characters $\chi_2\pmod{q_2}$ and $\chi_2'\pmod{q_2'}$ such that both $\chi_1\overline{\chi_2} := \chi$ and $\chi_1\overline{\chi_2'} := \chi'$ are induced by $\psi$. It suffices to show that $\chi_2(n) = \chi_2'(n)$ whenever $(n,q_2q_2')=1$, since this would imply that $\chi_2\chi_2'$ is the principal character, and thus $\chi_2=\chi_2'$ as they are both primitive. 

If $(n,q_2q_2')=1$, then we may find an integer $k$ such that $(n+kq_2q_2',q_1q_2q_2')=1$. Thus $\chi_1(n+kq_2q_2') \neq 0$ and $\chi(n+kq_2q_2') = \chi'(n+kq_2q_2')$. It follows that
\[ \chi_2(n) = \chi_2(n+kq_2q_2') = (\overline{\chi\chi_1})(n+kq_2q_2') = (\overline{\chi'\chi_1})(n+kq_2q_2') = \chi_2'(n+kq_2q_2') = \chi_2'(n). \]
This completes the proof of the claim.

It follows that the contribution to~\eqref{eq: LHS} from those $\chi_1,\chi_2$ with $\chi_1\overline{\chi_2} \in \Xi(\eta)$ is 
\[ \ll \eta \Psi(x,y) |\Xi^*(\eta)| \sum_{\chi} |b_{\chi}|^2. \]
We will show that $|\Xi^*(\eta)| \ll \eta^{-1/2}$ so that the result follows by summing over $\eta$ dyadically. 
We may assume that $\eta \geq Q^{-8}$, as the bound follows for smaller $\eta$ from  the trivial bound $|\Xi^*(\eta)| \leq Q^4$. 

We now use Proposition~\ref{prop:zerofree-to-bound} to show that if $\chi \in \Xi(\eta)$ then $L(s,\chi)$ has a zero in the region \eqref{0freeregion} for suitable values of $\varepsilon$ and $T$. This would imply that $L(s,\psi)$ has zero in the region~\eqref{0freeregion} for any $\psi \in \Xi^*(\eta)$.

For the purpose of contradiction, let's assume that $\chi \in \Xi(\eta)$ and $L(s,\chi)$ has no zero in  \eqref{0freeregion} with $T = Q^{500}$. We wish to verify the hypotheses in \eqref{eq:hyp-epsilon-H} and~\eqref{eq:hyp-epsilon-H-2}. The  upper bound on $T$ in~\eqref{eq:hyp-epsilon-H} follows from the definition of $Q$. Now $r\leq q\leq Q^2$ and so the first alternative of  \eqref{eq:hyp-epsilon-H-2} follows by selecting $c$ so that $502c\kappa\leq 1$. We define 
\[ 
\ee = \max\left(\frac{2\kappa}{\log y}, \frac{12(\log\eta^{-1} + \log\log x)}{\log x}\right) .
\]
Since $\log\eta^{-1} \leq 8\log Q$ and $\log\log x \leq \log Q$, 
we have the upper bound
\[ 
\ee \leq \max\left(\frac{2\kappa}{\log y}, \frac{108\log Q}{\log x}\right) \ll \frac 1 {\log\log x} ,
 \]
so that $\ee = o(1)$ and $y^{\ee} \ll Q^{108}$. The first hypothesis in \eqref{eq:hyp-epsilon-H} follows immediately. Finally, by selecting $C$ so that $cC\geq 2$ we guarantee that $Q\geq (\log x)^2$, so that the lower bound on $T$ in \eqref{eq:hyp-epsilon-H} follows easily.

Now $\ee \geq 12(\log\eta^{-1} + \log\log x)/\log x$ so that $x^{-0.3\ee} \leq (\eta/\log x)^{3.6}$, and therefore
\[ 
\sqrt{(\log x)(\log y)}(x^{-0.3\ee} \log T + T^{-0.02}) \leq \eta^{3.6}(\log x)^{-1.6} + Q^{-10}\log x. 
\]
Now $Q^{-10}\log x = o(Q^{-8}) = o(\eta)$, as $Q\geq (\log x)^2$. Therefore Proposition~\ref{prop:zerofree-to-bound} implies that  
\[ 
\left|\sum_{\substack{n \leq x \\ P(n) \leq y}} \chi(n) \right| \ll o(\eta \Psi(x,y)),
\]
contradicting the definition of $\Xi(\eta)$.

By Proposition~\ref{prop:log-zerofree} we now deduce (remembering that characters in $\Xi^*(\eta)$ have conductors at most $Q^2$) that
\[ |\Xi^*(\eta)| \ll (Q^4 T)^{\frac 52 \ee} = Q^{1260\ee}  \ll \eta^{-1/2}, \]
which completes the proof.
\end{proof}

\begin{remark}\label{rmk:LS-GRH}
Assuming the Riemann hypothesis for Dirichlet $L$-functions, by Proposition~1 of~\cite{Har}, we have the bound~$\sum_{n\leq x, P(n)\leq y} \chi(n) = O(x^{1-c})$ uniformly for~$\chi$ non-principal~$\mod{q}$,~$q\leq x^c$, $(\log x)^C \leq y \leq x^c$, for some absolute constants~$C, c>0$. This implies an upper bound
$$ \ll \Big(\Psi(x, y) + Q^2x^{1-c}\Big) \sum_{\chi} |b_\chi|^2 $$
for~\eqref{eq: LHS}, for all~$Q\leq x^c$, and Theorem~\ref{thm:large-sieve-smooth} would hold with~$Q=x^{c/3}$ and~$C$ large enough.
\end{remark}

\subsection{A variant of Theorem~\ref{thm:large-sieve-smooth}}

We may extend the range to $Q = x^c$ in Theorem~\ref{thm:large-sieve-smooth} unconditionally if we insert some weights that reduce the effects of characters with large conductor.

\begin{theorem}\label{thm:large-sieve-smooth-variant}
There exists $C,c>0$ such that the following statement holds. Let $(\log x)^{C} \leq y \leq x$ be large, and let $Q = x^c$.
For any sequence $\{a_n\}$ we have
\[ \sum_{q \leq Q} \frac{1}{q^{1/2}} \sum^*_{\chi\pmod q} \left| \sum_{\substack{n \leq x \\ P(n) \leq y}} a_n \chi(n) \right|^2 \ll \Psi(x,y) \cdot \sum_{\substack{n \leq x \\ P(n) \leq y}} |a_n|^2. \]
\end{theorem}

\begin{proof}
The proof is similar as the proof of Theorem~\ref{thm:large-sieve-smooth}. We begin by passing to its dual form, so that we need to prove that
\[ \sum_{\substack{n \leq x \\ P(n) \leq y}} \left| \sum_{q \leq Q} \frac{1}{q^{1/4}} \sum^*_{\chi\pmod q} b_{\chi} \chi(n) \right|^2 \ll \Psi(x,y) \cdot \sum_{q \leq Q} \frac{1}{q^{1/4}} \sum^*_{\chi \pmod q} |b_{\chi}|^2 \]
for any sequence $\{b_{\chi}\}$, where the summation is over all primitive characters $\chi\pmod{q}$ with $q \leq Q$. Expanding the square, we can bound the left hand side above by
\begin{equation}\label{eq:variant-1} 
\sum_{\chi_1\pmod{q_1}} \sum_{\chi_2\pmod{q_2}} \frac{|b_{\chi_1} b_{\chi_2}|}{(q_1q_2)^{1/4}} \left|\sum_{\substack{n \leq x \\ P(n) \leq y}} (\chi_1\overline{\chi_2})(n)\right|.
\end{equation}
For $\eta \in (0,1/2]$, define $\Xi(\eta)$ and $\Xi^*(\eta)$ as in the proof of Theorem~\ref{thm:large-sieve-smooth}. If $\chi_1\overline{\chi_2} \in \Xi(\eta)$ so that it is induced by some $\psi\pmod{r} \in \Xi^*(\eta)$, then
\[ \frac{|b_{\chi_1}b_{\chi_2}|}{(q_1q_2)^{1/4}} \leq \frac{1}{r^{1/8}}\left(\frac{|b_{\chi_1}|^2}{q_1^{1/4}} + \frac{|b_{\chi_2}|^2}{q_2^{1/4}}\right). \]
Thus the contribution to~\eqref{eq:variant-1} from those $\chi_1,\chi_2$ with $\chi_1\overline{\chi_2} \in \Xi(\eta)$ is
\begin{align*} 
& \ll \eta \Psi(x,y) \sum_{\psi\pmod{r} \in \Xi^*(\eta)} \frac{1}{r^{1/8}} \sum_{\chi_1\pmod{q_1}} \frac{|b_{\chi_1}|^2}{q_1^{1/4}} \sum_{\substack{\chi_2 \\ \chi_1\overline{\chi_2}\text{ induced by }\psi}} 1 \\
& \ll \eta \Psi(x,y) \left(\sum_{\psi\pmod{r} \in \Xi^*(\eta)} \frac{1}{r^{1/8}}\right) \left(\sum_{\chi\pmod{q}} \frac{|b_{\chi}|^2}{q^{1/4}}\right). 
\end{align*}
Hence it suffices to show that
\[ \sum_{\psi\pmod{r} \in \Xi^*(\eta)} \frac{1}{r^{1/8}} \ll \eta^{-1/2}, \]
and then the conclusion follows after dyadically summing over $\eta$. For $1 \leq R \leq Q^2$, let $\Xi(\eta, R)$ and $\Xi^*(\eta, R)$ be the set of characters in $\Xi(\eta)$ and $\Xi^*(\eta)$ with conductors $\sim R$, respectively. Thus it suffices to show that
\[ |\Xi^*(\eta, R)| \ll \eta^{-1/2} R^{1/9}, \]
for each $\eta \in (0,1/2]$ and $1 \leq R \leq Q^2$. We may assume that $\eta \geq R^{-4}$, since otherwise the trivial bound $|\Xi^*(\eta,R)| \ll R^2$ suffices. From now on fix such $\eta$ and $R$.

We now use Proposition~\ref{prop:zerofree-to-bound} to show that if $\chi \in \Xi(\eta, R)$ then $L(s,\chi)$ has a zero in the region \eqref{0freeregion} for suitable values of $\varepsilon$ and $T$. This would imply that $L(s,\psi)$ has zero in the region~\eqref{0freeregion} for any $\psi \in \Xi^*(\eta, R)$.

Set $T = (\eta^{-1}\log x)^{60}$. If $R \leq (\log x)^{10}$ (say), then the first alternative of~\eqref{eq:hyp-epsilon-H-2} holds because
\[ (2TR)^{\kappa} \leq (\log x)^{2500\kappa} \leq y, \]
provided that $C \geq 2500\kappa$. In this case we will set $\ee$ to be exactly the same as before:
\[ \ee := \max\left(\frac{2\kappa}{\log y}, \frac{12(\log\eta^{-1}+\log\log x)}{\log x}\right), \ \ \text{if } R \leq (\log x)^{10}.  \]
Then~\eqref{eq:hyp-epsilon-H} can be easily verified, and the contrapositive of Proposition~\ref{prop:zerofree-to-bound} implies that $L(s,\chi)$ has a zero in the region~\eqref{0freeregion} whenever $\chi \in \Xi(\eta, R)$. Hence by Proposition~\ref{prop:log-zerofree} we have
\[ |\Xi^*(\eta, R)| \ll (R^2T)^{\frac 52 \ee} \ll \eta^{-150\ee} R^{5\ee} (\log x)^{150\ee} \ll \eta^{-1/2} R^{1/9}, \]
since $\ee \ll 1/\log\log x$ in this case.

It remains to consider the case when $(\log x)^{10} \leq R \leq Q^2$. We set $T$ as above, and we will now set
\[ \ee := \max\left(\frac{2\kappa}{\log y}, \frac{12(\log\eta^{-1}+\log\log x)}{\log x}, \frac{50\log\log x}{\log y}\right), \]
so that the second alternative in~\eqref{eq:hyp-epsilon-H-2} is satisfied. One can still easily verify~\eqref{eq:hyp-epsilon-H}, and thus Propositions~\ref{prop:zerofree-to-bound} and~\ref{prop:log-zerofree} combine to give
\[ |\Xi^*(\eta, R)| \ll (R^2T)^{\frac 52\ee} \ll \eta^{-150\ee} R^{5\ee} (\log x)^{150\ee} \ll \eta^{-1/2} R^{1/9},  \]
since $\ee \leq 1/300$ (by choosing $C$ large enough) and $\log x \leq R^{1/10}$. This completes the proof.
\end{proof}

Examining the proof, one easily sees that the weight $1/q^{1/2}$ can be replaced by $1/q^{\sigma}$ for any constant $\sigma>0$, and the statement remains true provided that $C$ is large enough in terms of $\sigma$. For our purposes, any exponent strictly smaller than $1$ suffices.

\section{Bounding the number of exceptional characters}

\begin{corollary}\label{cor:exceptional-good-bound}
There exist $C,c>0$ such that the following statement holds.
Let $(\log x)^{C} \leq y \leq x^{1/4}$ be large. 
Let $\{a_n\}$ be an arbitrary $1$-bounded sequence. For $B \geq 0$, let $\Xi(B)$ be the set of primitive characters $\chi\pmod{r}$ with $r \leq Q$ where
\[ Q := \min(y^c, \exp(c \log x/\log\log x), \]
such that there exists $x^{1/4} < X \leq x$ for which
\[ \left| \sum_{\substack{n \leq X \\ P(n) \leq y}} a_n \chi(n) \right| \geq \frac{\Psi(X,y)}{(u\log u)^4 (\log x)^B}. \]
Then $|\Xi(B)| \ll (\log x)^{3B+13}$. 
\end{corollary}

\begin{proof} Let $T=(u\log u)^4 (\log x)^B$. We begin by partitioning the interval $[x^{1/4},x]$ using a sequence $x^{1/4} = X_0 < X_1 < \cdots < X_{J-1} < X_J = x$ with $J \asymp T \log x$, such that $ X_{j+1} - X_j \asymp \varepsilon X_j/T$, 
for some fixed small enough $\ee>0$, for each $0 \leq j < J$. 

For each $\chi \in \Xi(B)$, there exists some $0 \leq j < J$ for which
\[ 
\left|\sum_{n \leq X_j} a_n \chi(n)\right| \geq \frac{\Psi(X_j,y)}{T} - \sum_{\substack{X_j < n \leq X_{j+1} \\ P(n) \leq y}} 1. 
\]
Corollary 2 of \cite{Hil} implies a good  upper bound for smooth numbers in short intervals: For any fixed $\kappa>0$,
\begin{equation} \label{SmoothShorts}
\Psi(x+\frac xT,y)-\Psi(x,y) \ll_\kappa \frac {\Psi(x,y)}T \text{  for  } 1\leq T\leq \min \{ y^\kappa, x\}.
\end{equation}
In our case $u\log u\leq \log x$ so that $T\leq (\log x)^{B+4}$, so the hypothesis here is satisfied, and we therefore have
\[
\sum_{\substack{X_j < n \leq X_{j+1} \\ P(n) \leq y}} 1 = \Psi(X_{j+1},y)-\Psi(X_j,y)\ll \varepsilon \frac {\Psi(X_j,y)}T .
\] 
By choosing $\varepsilon$ sufficiently small we deduce that 
\begin{equation} \label{eq: a_n lower bd}
\left|\sum_{n \leq X_j} a_n \chi(n)\right| \geq \frac{\Psi(X_j,y)}{2T}  . 
 \end{equation}

We deduce that there  exists some $0 \leq j < J$ such that \eqref{eq: a_n lower bd}
holds for at least $|\Xi(B)|/J$ characters $\chi \in \Xi(B)$. Therefore
\[ 
\sum_{\chi \in \Xi(B)} \left|\sum_{\substack{n \leq X_j \\ P(n) \leq y}} a_n \chi(n) \right|^2 \geq \frac{ |\Xi(B)| }J \cdot \frac{\Psi(X_j,y)^2}{4T^2} \gg  \frac{ |\Xi(B)|  \Psi(X_j,y)^2}{ T^3 \log x} 
\]
On the other hand,  Theorem~\ref{thm:large-sieve-smooth} implies that 
\[ 
\sum_{\chi \in \Xi(B)} \left|\sum_{\substack{n \leq X_j \\ P(n) \leq y}} a_n \chi(n) \right|^2  \leq \sum_{r \leq Q} \sum^*_{\chi\pmod{r}} \left|\sum_{\substack{n \leq X_j \\ P(n) \leq y}} a_n \chi(n)\right|^2 \ll \Psi(X_j,y)^2 ,
\]
and therefore $|\Xi(B)| \ll T^3 \log x = (u\log u)^{12} (\log x)^{3B+1} \ll (\log x)^{3B+13}$, as claimed.
\end{proof}

We also record the following variant which gives a weighted count of exceptional characters, but now with the wider range $Q = x^c$.

\begin{corollary}\label{cor:exceptional-good-bound-variant}
There exist $C,c>0$ such that the following statement holds.
Let $(\log x)^{C} \leq y \leq x^{1/4}$ be large. 
Let $\{a_n\}$ be an arbitrary $1$-bounded sequence. For $B \geq 0$, let $\Xi(B)$ be the set of primitive characters $\chi\pmod{r}$ with $r \leq Q := x^c$, such that there exists $x^{1/4} < X \leq x$ for which
\[ \left| \sum_{\substack{n \leq X \\ P(n) \leq y}} a_n \chi(n) \right| \geq \frac{\Psi(X,y)}{(u\log u)^4 (\log x)^B}. \]
Then 
\[ \sum_{\psi\pmod{q} \in \Xi(B)} \frac{1}{q^{1/2}} \ll (\log x)^{3B+13}. \]
\end{corollary}

\begin{proof}
The proof is the same as above, except that one considers the weighted sum 
\[ \sum_{\chi\pmod{q} \in \Xi(B)} \frac{1}{q^{1/2}} \left|\sum_{\substack{n \leq X_j \\ P(n) \leq y}} a_n \chi(n) \right|^2, \]
and use Theorem~\ref{thm:large-sieve-smooth-variant} instead of Theorem~\ref{thm:large-sieve-smooth} in the last step.
\end{proof}


\end{document}